\documentclass{amsart}
\usepackage{amsmath}
\usepackage{amsfonts}
\usepackage{cite}

\theoremstyle{plain}

\newtheorem{definition}{Definition}
\newtheorem{example}{Example}

\newtheorem{proposition}{Proposition}

\numberwithin{equation}{section}
\input{tcilatex}

\begin{document}
\title[Some special curves in the unit tangent bundles of surfaces]{Some
special curves in the unit tangent bundles of surfaces}
\author{Murat Altunba\c{s}}
\address{Department of Mathematics \\
Faculty of Sciences and Arts\\
Erzincan Binali Yildirim University\\
24100 Erzincan, Turkey}
\email{maltunbas@erzincan.edu.tr}
\date{}

\begin{abstract}
The aim of this paper is to give some characterizations for $N-$Legendre and 
$N-$slant curves in the unit tangent bundles of surfaces endowed with
natural diagonal lifted structures.


\end{abstract}

\maketitle

\section{Introduction}

In studies on curves in the unit tangent (sphere) bundles, researchers
generally consider the standard contact metric structure which is obtained
by endowing the bundle with the induced Sasaki metric. For examples, in \cite%
{Berndt} Berndt \textit{et al.} studied the geodesics, in \cite{Inoguchi1}
and \cite{Inoguchi2} Inoguchi and Munteanu investigated the magnetic curves,
in \cite{Hou} Hou and Sun considered the slant geodesics and in \cite%
{Hathout} Hathout \textit{et al. }discussed the $N-$Legendre and $N-$slant
curves of the unit tangent bundles with respect to this metric structure.
However, some other contact metric structures can be defined on the unit
tangent bundles. One of them is introduced by Druta-Romaniuc and Oproiu on
tangent (sphere) bundles and called natural diagonal structure in \cite%
{Opriou2}. In this paper, they found conditions under which the tangent
sphere bundles are Einstein. In their further works, they had conditions
under which the tangent sphere bundles are $\eta -$Einstein and obtained
some results for curvatures of the tangent sphere bundles (see \cite{Opriou1}
and \cite{Opriou3}).

In this paper, $N-$Legendre and $N-$slant curves are studied in the unit
tangent bundles of surfaces with natural diagonal structures and some
results are given when the surface is considered to be a sphere.

\section{Preliminiaries}

In this section, we give a brief introduction to natural diagonal
structures, for further information see \cite{Opriou1}. Let $(M,g)$ be a
smooth $n-$dimensional Riemannian manifold and $\pi :TM\rightarrow M$ be its
tangent bundle. Let $(x^{i},u^{i})_{(i=1,...,n)}$ be the locally coordinate
systems on the tangent bundle $TM$. The natural diagonal lift metric $g^{d}$
is defined as follows: 
\begin{eqnarray}
{}g^{d}({}X^{h},{}Y^{h}) &=&c_{1}{}g(X,Y)+d_{1}g(X,u)g(Y,u),  \label{1} \\
g^{d}({}X^{v},{}Y^{h}) &=&{}g^{d}({}X^{h},{}Y^{v})=0,  \notag \\
g^{d}({}X^{v},{}Y^{v}) &=&{}c_{2}{}g(X,Y)+d_{2}g(X,u)g(Y,u),  \notag
\end{eqnarray}%
for every vector fields $X,Y$ on $M$ and every tangent vector $u,$ where $%
t=g(u,u)/2$ and $c_{1},c_{2},d_{1},d_{2}$ are smooth functions of $t.$ The
conditions for $g^{d}$ to be positive are $c_{1}>0,\ c_{2}>0,\
c_{1}+2td_{1}>0,\ c_{2}+2td_{2}>0$ for every $t\geq 0.\ $Here, $X^{h}=X^{i}%
\frac{\partial }{\partial x^{i}}-X^{i}u^{j}\Gamma _{ij}^{k}\frac{\partial }{%
\partial u^{k}}$ and $X^{v}=X^{i}\frac{\partial }{\partial u^{i}}$ are the
horizontal and the vertical lifts of $X$ at $(x,u)$ with respect to
Levi-Civita connection $\nabla $ of $g$ respectively, where $\left\{ \Gamma
_{ij}^{k}\right\} $ are the Christoffel symbols of $\nabla $.

Let us define an $(1,1)-$tensor field $J$ on $TM$ as follows:%
\begin{eqnarray}
JX^{h} &=&a_{1}X^{v}+b_{1}g(X,u)u^{v},  \label{2} \\
JX^{v} &=&-a_{2}X^{h}-b_{2}g(X,u)u^{h},  \notag
\end{eqnarray}%
for every vector field $X$ on $M$, where $a_{1},a_{2},b_{1},b_{2}$ are
smooth functions of $t.$We note that the $(1,1)-$tensor field $J$ given by
the relations (\ref{2}) defines an almost complex structure on the tangent
bundle if and only if $a_{2}=1/a_{1}$ and $%
b_{2}=-b_{1}/[a_{1}(a_{1}+2tb_{1})]\ $(see \cite{Opriou3}).

We know that the unit tangent bundle $T_{1}M=\{u\in TM:g(u,u)=1\}$ of a
Riemannian manifold $M$ is a $(2n-1)-$dimensional submanifold of $TM.$ The
canonical vector field $u^{v}$ is normal to $T_{1}M.$ The horizontal lift of
any vector field on $M$ is tangent to $T_{1}M$, but the vertical lift is not
always tangent to $T_{1}M.$ The tangential lift of a vector field $X$ of $M$
is defined by $X^{t}=X^{v}-g(X,u)u^{v}.$ Hence, we write the Lie algebra of $%
C^{\infty }$ vector fields on $T_{1}M$ as $\chi
(T_{1}M)=\{X^{h}+Y^{t}:X,Y\in \chi (M)\}$ \cite{Boeckx}. The induced
Riemannian metric $g_{1}^{d}$ on $T_{1}M$ from (\ref{1}) is uniquely
determined by 
\begin{eqnarray}
{}g_{1}^{d}({}X^{h},{}Y^{h}) &=&c_{1}{}g(X,Y)+d_{1}g(X,u)g(Y,u),  \label{3}
\\
g_{1}^{d}({}X^{v},{}Y^{h}) &=&{}g_{1}^{d}({}X^{h},{}Y^{v})=0,  \notag \\
g_{1}^{d}({}X^{v},{}Y^{v}) &=&{}c_{2}{}[g(X,Y)-g(X,u)g(Y,u)],  \notag
\end{eqnarray}%
for every vector fields $X,Y$ on $M$ and every tangent vector $u,$ where $%
c_{1},d_{1},c_{2}$ are constants. The conditions for ${}g_{1}^{d}$ to be
positive are $c_{1}>0,$ $c_{2}>0,$ $c_{1}+d_{1}>0$ \cite{Opriou2}$.$

Remark that the functions $c_{1},d_{1},c_{2}$ become constant, since in the
case of unit tangent bundle, the function $t$ becomes a constant equal to $%
\frac{1}{2}$.

In \cite{Opriou1}, it is proved that there is a contact metric structure $%
(\varphi _{1},\xi _{1},\eta _{1},g_{1})$ on $T_{1}M$ given by%
\begin{eqnarray}
\varphi _{1}(X^{h}) &=&a_{1}X^{t},\text{ }\varphi
_{1}(X^{t})=-a_{2}X^{h}+a_{2}g(X,u)u^{h},  \label{4} \\
\xi _{1} &=&\frac{1}{2\lambda \alpha }u^{h},\text{ }\eta _{1}(X^{t})=0,\text{
}\eta _{1}(X^{h})=2\alpha \lambda g(X,u),\text{ }g_{1}=\alpha g_{1}^{d}, 
\notag
\end{eqnarray}%
for every vector fields $X,Y$ on $M$ and every tangent vector $u,$ where $%
\lambda $ is a scalar, $\alpha =\frac{c_{1}+d_{1}}{4\lambda ^{2}}\ $and $%
a_{1}$ and $a_{2}$ are the functions defined in (\ref{2})$.$ This contact
metric structure is called natural diagonal structure. Furthermore, $%
(T_{1}M,\varphi _{1},\xi _{1},\eta _{1},g_{1})$ is Sasakian if and only if $%
M $ has constant sectional curvature $K=a_{1}^{2}$ \cite{Opriou1}.

The Levi-Civita connection $\nabla _{1}$ of $(T_{1}M,g_{1})$ satisfies the
following relations:%
\begin{eqnarray}
&&  \label{5} \\
\nabla _{1\text{ }X^{h}}Y^{h} &=&(\nabla _{X}Y)^{h}-\frac{1}{2}(R(X,Y)u)^{t}-%
\frac{d_{1}}{2c_{2}}[g(X,u)Y^{t}-g(Y,u)X^{t}],  \notag \\
\nabla _{1\text{ }X^{h}}Y^{t} &=&(\nabla _{X}Y)^{t}-\frac{c_{2}}{2c_{1}}%
(R(Y,u)X)^{h}+\frac{d_{1}}{2c_{1}}g(X,u)Y^{h}+\frac{d_{1}}{2(c_{1}+d_{1})}%
g(X,Y)u^{h}  \notag \\
&&-\frac{d_{1}(2c_{1}+d_{1})}{2c_{1}(c_{1}+d_{1})}g(X,u)g(Y,u)u^{h}-\frac{%
c_{2}d_{1}}{2c_{1}(c_{1}+d_{1})}g(Y,R(X,u)u)u^{h},  \notag \\
\nabla _{1\text{ }X^{t}}Y^{h} &=&-\frac{c_{2}}{2c_{1}}(R(X,u)Y)^{h}+\frac{%
d_{1}}{2c_{1}}g(Y,u)X^{h}+\frac{d_{1}}{2(c_{1}+d_{1})}g(X,Y)u^{h}  \notag \\
&&-\frac{d_{1}(2c_{1}+d_{1})}{2c_{1}(c_{1}+d_{1})}g(X,u)g(Y,u)u^{h}-\frac{%
c_{2}d_{1}}{2c_{1}(c_{1}+d_{1})}g(X,R(Y,u)u)u^{h},  \notag \\
\nabla _{1\text{ }X^{t}}Y^{t} &=&-g(Y,u)X^{t},  \notag
\end{eqnarray}%
for every vector fields $X,Y$ on $M$ and every tangent vector $u,$ where $%
\nabla $ and $R$ denote the Levi-Civita connection and the curvature tensor
of $(M,g),$ respectively \cite{Opriou1}.

\section{$N-$Legendre and $N-$slant curves}

Let $(M,g)$ be a surface and let $\gamma :I\subset 
\mathbb{R}
\rightarrow M$ be a curve on $M.$ Assume that $\tilde{\gamma}(s)=(\gamma
(s),X(s))$ is a curve on $(T_{1}M,g_{1},\varphi _{1},\xi _{1},\eta _{1}),$
where the contact metric structure is given by (\ref{4}). We have four kinds
of curves which are defined below.

\begin{definition}
\label{def1} \cite{Hou} Let $\gamma $ be a curve in an almost contact metric
manifold $(M,g,\varphi ,\xi ,\eta )$. The curve $\gamma $ is called
Legendrian (resp. slant) if the angle between the tangent vector field $T$
of $\gamma $ and $\xi $ is $\frac{\pi }{2}$ (resp. $[0,\pi ]-\{\pi /2\}),$
i.e. $g(T,\xi )=0$ (resp. $g(T,\xi )=c),$ where $c$ is a non-zero constant.
\end{definition}

\begin{definition}
\label{def2} \cite{Hathout} Let $\gamma $ be a curve in an almost contact
metric manifold $(M,g,\varphi ,\xi ,\eta )$. The curve $\gamma $ is called
N-Legendre (resp. N-slant) if the angle between the normal vector field $N$
of $\gamma $ and $\xi $ is $\pi /2$ (resp. $[0,\pi ]-\{\pi /2\}$), i.e. $%
g(N,\xi )=0$ (resp. $g(N,\xi )=c),$ where $c$ is a non-zero constant.
\end{definition}

Suppose that $\tilde{\gamma}(s)$ is parameterized by the arc-length and
denote the Frenet apparatus of $\tilde{\gamma}(s)$ by $(\hat{T},\tilde{N},%
\tilde{B},\tilde{\kappa},\tilde{\tau})$. Then,%
\begin{eqnarray}
\tilde{T}(s) &=&\frac{d\gamma ^{i}}{ds}\frac{\partial }{\partial x^{i}}+%
\frac{dX^{i}}{ds}\frac{\partial }{\partial u^{i}}  \label{6} \\
&=&\frac{d\gamma ^{i}}{ds}(\frac{\partial }{\partial x^{i}})^{h}(\tilde{%
\gamma}(s))+(\frac{dX^{i}}{ds}+\frac{d\gamma ^{j}}{ds}X^{k}\Gamma _{jk}^{i})%
\frac{\partial }{\partial u^{i}}(\tilde{\gamma}(s))  \notag \\
&=&(E^{h}+(\nabla _{E}X)^{t})(\tilde{\gamma}(s)),  \notag
\end{eqnarray}%
where $E=\gamma ^{\prime }(s).$

Let $\theta $ be the angle between $\tilde{T}$ and $\xi _{1}$. From
equations (\ref{3}) and (\ref{4}), we have 
\begin{equation}
\frac{g_{1}(\tilde{T},\xi _{1})}{\left\vert \tilde{T}\right\vert \left\vert
\xi _{1}\right\vert }=\cos \theta =\sqrt{c_{1}+d_{1}}g(E,X).  \label{7}
\end{equation}%
If we differentiate both side of equation (\ref{7}) with respect to $s,$ and
use equations (\ref{4}), (\ref{5}) and (\ref{6}), we have 
\begin{eqnarray*}
\frac{d}{ds}g_{1}(\tilde{T},\xi _{1}) &=&g_{1}(\nabla _{1}\text{ }_{\tilde{T}%
(s)}\tilde{T},\xi _{1})+g_{1}(\tilde{T},\nabla _{1}\text{ }_{\tilde{T}%
(s)}\xi _{1}) \\
&=&\tilde{\kappa}g_{1}(\tilde{N},\xi _{1})+\frac{1}{2\lambda \alpha }g_{1}(%
\tilde{T},\nabla _{1E^{h}}X^{h}+\nabla _{1(\nabla _{E}X)^{t}}X^{h}) \\
&=&\tilde{\kappa}g_{1}(\tilde{N},\xi _{1})+\frac{1}{2\lambda \alpha ^{2}}%
[(c_{1}+d_{1})g(E,\nabla _{E}X)-c_{2}R(E,X,X,\nabla _{E}X)] \\
&=&-\left\vert \xi _{1}\right\vert \theta ^{\prime }\sin \theta .
\end{eqnarray*}%
So,%
\begin{equation}
g_{1}(N,\xi _{1})=\frac{1}{2\lambda \alpha ^{2}\tilde{\kappa}}\left(
c_{2}R(E,X,X,\nabla _{E}X)-(c_{1}+d_{1})g(E,\nabla _{E}X)\right) -\left\vert
\xi _{1}\right\vert \frac{\theta ^{\prime }\sin \theta }{\tilde{\kappa}},
\label{8}
\end{equation}%
where $\theta ^{\prime }=\frac{d\theta }{ds}$ and $\xi _{1}=\frac{1}{%
2\lambda \alpha }X^{h}$.

Let $(T,N)$ be a Frenet frame on $\gamma .$ From equation (\ref{7}), we get
the following 
\begin{equation}
X=\frac{\lambda }{r\sqrt{c_{1}+d_{1}}}\cos \theta T+\beta N,  \label{9}
\end{equation}%
for a smooth function $\beta ,$ where $r=\left\Vert E\right\Vert .$ Since $X$
is a unit vector, we have 
\begin{equation*}
\frac{\lambda ^{2}}{(c_{1}+d_{1})r^{2}}\cos ^{2}\theta +\beta ^{2}=1,
\end{equation*}%
and%
\begin{equation}
\beta =\pm \frac{1}{r}\sqrt{r^{2}-(\frac{\lambda }{\sqrt{c_{1}+d_{1}}}%
)^{2}\cos ^{2}\theta }.  \label{10}
\end{equation}%
Differentiating equation (\ref{9}) with respect to $s$, we derive%
\begin{eqnarray}
\nabla _{E}X &=&\frac{1}{\sqrt{c_{1}+d_{1}}}(\frac{\cos \theta }{r})^{\prime
}T+\frac{\kappa \cos \theta }{\sqrt{c_{1}+d_{1}}}N+\beta ^{\prime }N-r\beta
\kappa T  \label{11} \\
&=&((\frac{\cos \theta }{r\sqrt{c_{1}+d_{1}}})^{\prime }-r\beta \kappa )T+(%
\frac{\kappa \cos \theta }{\sqrt{c_{1}+d_{1}}}+\beta ^{\prime })N.  \notag
\end{eqnarray}%
Equations (\ref{7}) and (\ref{11}), and orthogonality of the vectors $X$ and 
$\nabla _{E}X$ give us 
\begin{equation}
E=\frac{\cos \theta }{\sqrt{c_{1}+d_{1}}}X+\frac{r}{g(\nabla _{E}X,\nabla
_{E}X)}\left( (\frac{\cos \theta }{r\sqrt{c_{1}+d_{1}}})^{\prime }-r\beta
\kappa \right) \nabla _{E}X.  \label{12}
\end{equation}%
Using the last expression, we can write%
\begin{eqnarray}
R(E,X,X,\nabla _{E}X) &=&r\left( (\frac{\cos \theta }{r\sqrt{c_{1}+d_{1}}}%
)^{\prime }-r\beta \kappa \right) \frac{R(\nabla _{E}X,X,X,\nabla _{E}X)}{%
g(\nabla _{E}X,\nabla _{E}X)}  \notag \\
&=&r\left( (\frac{\cos \theta }{r\sqrt{c_{1}+d_{1}}})^{\prime }-r\beta
\kappa \right) K(s),  \label{123}
\end{eqnarray}%
where $K(s)$ is the sectional curvature of $M.$ Putting the equations (\ref%
{10})-(\ref{123}) in (\ref{8}), we state the following equation%
\begin{eqnarray}
&&  \label{13} \\
g_{1}(\tilde{N},\xi _{1}) &=&\frac{r(c_{2}K(s)-(c_{1}+d_{1}))}{2\lambda
\alpha ^{2}\tilde{\kappa}}\left( (\frac{\cos \theta }{r\sqrt{c_{1}+d_{1}}}%
)^{\prime }\pm r\kappa \sqrt{r^{2}-(\frac{\lambda }{\sqrt{c_{1}+d_{1}}}%
)^{2}\cos ^{2}\theta }\right)   \notag \\
&&-\left\vert \xi _{1}\right\vert \frac{\theta ^{\prime }\sin \theta }{%
\tilde{\kappa}}.  \notag
\end{eqnarray}%
Now we can prove the following propositions.

\begin{proposition}
\bigskip Let $T_{1}S^{2}$ be the unit tangent bundle of the unit sphere $%
S^{2}$ with the natural diagonal metric structure given by (\ref{4}) such
that $c_{2}=c_{1}+d_{1}.$ Then all Legendre and slant curves are $\tilde{N}-$%
Legendre curves.
\end{proposition}

\begin{proof}
Let $\tilde{\gamma}(s)=(\gamma (s),X(s))$ be a Legendre or a slant curve
with arc-parameter in the contact metric manifold $T_{1}S^{2}$. Since the
sectional curvature of the unit sphere $K$ is equal to $1$, from Definition %
\ref{def1} and equation (\ref{13}) and under the assumption $%
c_{2}=c_{1}+d_{1}$, we get 
\begin{equation*}
g_{1}(\tilde{N},\xi _{1})=0.
\end{equation*}%
This completes the proof.
\end{proof}

\begin{proposition}
\bigskip Let $T_{1}S^{2}$ be the unit tangent bundle of the unit sphere $%
S^{2}$ with the natural diagonal metric structure given by (\ref{4}) such
that $c_{2}=c_{1}+d_{1}$ and let $\tilde{\gamma}$ be a non-slant curve on $%
T_{1}S^{2}.$ Then $\tilde{\gamma}$ is an $\tilde{N}-$slant curve if the
angle $\theta $ satisfies the equation%
\begin{equation*}
\theta =\arccos c\dint \tilde{\kappa},
\end{equation*}%
where $c$ is a non-zero constant.
\end{proposition}

\begin{proof}
Let $\tilde{\gamma}(s)=(\gamma (s),X(s))$ be a non-slant curve with
arc-parameter in the contact metric manifold $T_{1}S^{2}$. Since the
sectional curvature of the unit sphere $K$ is equal to $1$, under the
assumption $c_{2}=c_{1}+d_{1},\ $Definition \ref{def1} and equation (\ref{13}%
) give us 
\begin{equation*}
g_{1}(N,\xi _{1})=-\frac{\theta ^{\prime }\sin \theta }{\tilde{\kappa}}=c%
\text{ \ }constant.
\end{equation*}%
So, 
\begin{equation*}
g_{1}(N,\xi _{1})=(\cos \theta )^{\prime }=c\tilde{\kappa}.\text{ }
\end{equation*}%
By solving the last differential equation, we get 
\begin{equation*}
\theta =\arccos c\dint \tilde{\kappa},
\end{equation*}%
which completes the proof.
\end{proof}

\begin{proposition}
\label{Prop}Let $M$ be a non-unit sphere whose constant sectional curvature
is $K=a_{1}^{\text{ }2}$. Suppose that $\tilde{\gamma}(s)=(\gamma (s),X(s))$
is a slant curve in $(T_{1}M,g_{1},\varphi _{1},\xi _{1},\eta _{1})$ such
that $c_{2}=c_{1}+d_{1}$ and $\gamma $ is a curve with constant velocity $%
r_{0}$. If the torsion $\tilde{\tau}$ of $\tilde{\gamma}$ equals to
sectional curvature $K$ of $M,$ then $\tilde{\gamma}$ is an $\tilde{N}-$%
Legendre (resp. $\tilde{N}-$slant) curve if and only if $\gamma $ is a
geodesic (resp. has a non-zero constant curvature $\kappa $).
\end{proposition}

\begin{proof}
Let $M$ be a sphere with constant sectional curvature is $K(s)=a_{1}^{\text{ 
}2}$ and let $c_{2}=c_{1}+d_{1}$. If the curve $\tilde{\gamma}(s)=(\gamma
(s),X(s))$ is a slant curve in $(T_{1}M,g_{1},\varphi _{1},\xi _{1},\eta
_{1})$ where $\gamma $ has constant velocity of $r_{0},$ then from (\ref{13}%
) we have 
\begin{equation*}
g_{1}(N,\xi _{1})=\frac{r_{0}c_{2}(K(s)-1)}{2\lambda \alpha ^{2}\tilde{\kappa%
}}\left( \pm \kappa r_{0}\sqrt{r_{0}^{2}-(\frac{\lambda }{\sqrt{c_{1}+d_{1}}}%
)^{2}\cos ^{2}\theta }\right) .
\end{equation*}%
We know that in a Sasakian 3-manifold, a curve $\sigma $ is slant if and
only if $(\tau _{\sigma }\pm 1)/\kappa _{\sigma }$ is a non-zero constant,
where $\tau _{\sigma }$ and $\kappa _{\sigma }$ are torsion and curvature of 
$\sigma $ respectively (see \cite{Cho}). If we assume that $\tilde{\tau}=K,$
from the above equation, we have%
\begin{equation*}
g_{1}(N,\xi _{1})=\frac{(\tilde{\tau}-1)}{\tilde{\kappa}}\frac{r_{0}c_{2}}{%
2\lambda \alpha ^{2}}\left( \pm r_{0}\sqrt{r_{0}^{2}-(\frac{\lambda }{\sqrt{%
c_{1}+d_{1}}})^{2}\cos ^{2}\theta }\right) \kappa =\bar{c}\kappa ,
\end{equation*}%
where $\bar{c}$ is a constant. So, it is clear that $\tilde{\gamma}$ is an $%
\tilde{N}-$Legendre (resp. $\tilde{N}-$slant) curve if and only if $\kappa =0
$ (resp. non-zero constant). This ends the proof.
\end{proof}

\begin{example}
Let $S^{2}$ be a non-unit sphere with radius $R.$ In this case, the
sectional curvature (Gaussian curvature) of $S^{2}$ equals to $\frac{1}{R^{2}%
}.\ $Under the assumptions in Proposition \ref{Prop}, the projection curves $%
\gamma $ of all slant and $\tilde{N}-$slant curves $\tilde{\gamma}$ in $%
T_{1}S^{2}$ are circles in $S^{2}$ when their Frenet apparatus are $(\tilde{T%
},\tilde{N},\tilde{B},\tilde{\kappa},\tilde{\tau}=\frac{1}{R^{2}}).$
\end{example}

\begin{proposition}
\label{Prop2}\bigskip Let $M$ be a non-unit sphere and $T_{1}M$ be the unit
tangent bundle of $M$ with the natural diagonal metric structure given by (%
\ref{4}) such that $c_{2}=c_{1}+d_{1}$.\ Suppose that $\tilde{\gamma}%
(s)=(\gamma (s),X(s))$ is a slant curve on $T_{1}M$ and $\gamma $ is a curve
with constant velocity $r_{0}$. Then the curve $\tilde{\gamma}$ is $\tilde{N}%
-$slant if and only if 
\begin{equation*}
\frac{(K-1)\kappa }{\widetilde{\kappa }}
\end{equation*}%
is a non-zero constant.
\end{proposition}

\begin{proof}
Let $M$ be a non-unit sphere $(K\neq 1)$. Assume that the curve $\tilde{%
\gamma}(s)=(\gamma (s),X(s))$ is a slant curve in $(T_{1}M,g_{1},\varphi
_{1},\xi _{1},\eta _{1})$ such that $c_{2}=c_{1}+d_{1},$ where $\gamma $ has
constant velocity $r_{0}.\ $Then from (\ref{13}) we get%
\begin{equation*}
g_{1}(N,\xi _{1})=\pm \frac{(K-1)}{\tilde{\kappa}}\frac{r_{0}c_{2}}{2\lambda
\alpha ^{2}}\left( \pm r_{0}\sqrt{r_{0}^{2}-(\frac{\lambda }{\sqrt{%
c_{1}+d_{1}}})^{2}\cos ^{2}\theta }\right) \kappa =\overline{c}\frac{%
(K-1)\kappa }{\tilde{\kappa}},
\end{equation*}%
where $\overline{c}$ is a non-zero constant. The proof follows from the
definition \ref{def2}.
\end{proof}

\begin{example}
Let $\tilde{\gamma}$ be an arbitrary slant curve in $T_{1}%
\mathbb{R}
^{2}$ and its projection curve $\gamma $ be a geodesic in $%
\mathbb{R}
^{2}.\ $Then under the assumptions in Proposition \ref{Prop2}, $\tilde{\gamma%
}$ is an $\tilde{N}-$Legendre curve. Clearly, if $\gamma $ is not geodesic,
then $\tilde{\gamma}$ is an $\tilde{N}-$slant curve if and only if $\frac{%
\kappa }{\widetilde{\kappa }}$ is a non-zero constant.
\end{example}

\begin{proposition}
Let $M$ be a non-unit sphere and $T_{1}M$ be the unit tangent bundle of $M$
with the natural diagonal metric structure given by (\ref{4}). Suppose that $%
\tilde{\gamma}(s)=(\gamma (s),X(s))$ is a non-slant curve in $T_{1}M$ and $%
\gamma $ is a curve with constant velocity $\frac{2\lambda }{c_{1}+d_{1}}$.
If the angle $\theta $ is linear (i.e. $\theta =es+f,$ $e$ and $f$ are
constants), then \newline
(1) $\tilde{\gamma}(s)$ is a $\tilde{N}-$Legendre curve if and only if 
\begin{equation}
\frac{c_{2}(K-1)}{2\alpha ^{2}\left\vert \xi _{1}\right\vert \sqrt{%
c_{1}+d_{1}}}(-\frac{e}{\lambda }\pm \frac{\kappa \lambda }{c_{1}+d_{1}})=e.
\label{14}
\end{equation}%
(2) $\tilde{\gamma}(s)$ is a $\tilde{N}-$slant curve if and only if%
\begin{equation}
\theta =\arcsin (\frac{\overline{c}\tilde{\kappa}}{-e(\frac{c_{2}(K-1)}{%
2\alpha ^{2}\lambda \sqrt{c_{1}+d_{1}}}+1)\pm \frac{\kappa \lambda c_{2}(K-1)%
}{2\alpha ^{2}(c_{1}+d_{1})^{3/2}})}).  \label{15}
\end{equation}%
where $\bar{c}$ is a non-zero constant.
\end{proposition}

\begin{proof}
(1) If the angle $\theta $ is linear, under the assumptions and from
equation (\ref{13}), we have 
\begin{equation*}
g_{1}(\tilde{N},\xi _{1})=\frac{c_{2}(K-1)}{2\alpha ^{2}\sqrt{c_{1}+d_{1}}%
\tilde{\kappa}}(-\frac{e\sin \theta }{\lambda }\pm \frac{\kappa \lambda }{%
c_{1}+d_{1}}\sin \theta )-\frac{e\sin \theta }{\tilde{\kappa}}\left\vert \xi
_{1}\right\vert .
\end{equation*}%
The proof follows from the condition $g_{1}(\tilde{N},\xi _{1})=0$ and
direct computations.\newline

(2) The above equation, the $\tilde{N}-$slant condition $g_{1}(\tilde{N},\xi
_{1})=c$ and direct computations give the proof.
\end{proof}

\end{document}